\documentclass[10pt,reqno]{amsart}

\usepackage{amsthm, mathrsfs, amsmath, amstext, amsxtra, amsfonts, dsfont, amssymb}
\usepackage{xcolor}
\usepackage{lmodern}
\usepackage[colorlinks, linkcolor=red, citecolor=blue, urlcolor=blue, pagebackref, hypertexnames=false]{hyperref}

\usepackage{geometry}
\newcommand{\R}{\mathbb R}
\newcommand{\C}{\mathbb C}

\newcommand{\eps}{\epsilon}

\newcommand{\re}[1]{\mbox{Re} \ #1} 
\newcommand{\im}[1]{\mbox{Im} \ #1} 
 
\newcommand{\imem}[1]{\mbox{\emph{Im}} \ #1}

\newcommand{\defendproof}{\hfill $\Box$} 

\newtheorem{theorem}{Theorem}[section]
\newtheorem{lemma}[theorem]{Lemma} 
\newtheorem{proposition}[theorem]{Proposition}
 
\theoremstyle{definition}
\newtheorem{definition}[theorem]{Definition}
\newtheorem{remark}[theorem]{Remark}

\title[Blow-up dynamics NLS inverse-square potential]{Mass concentration and characterization of finite time blow-up solutions for the nonlinear Schr\"odinger equation with inverse-square potential} 

\author{Abdelwahab Bensouilah}
\address{Laboratoire Paul Painlev\'e (U.M.R. CNRS 8524), U.F.R. de Math\'ematiques, Universit\'e Lille 1, 59655 Villeneuve d'Ascq Cedex, France}
\email{ai.bensouilah@math.univ-lille1.fr}

\author{Van Duong Dinh}
\address{Institut de Mathematiques de Toulouse UMR5219, Universit\'e de Toulouse CNRS, 31062 Toulouse Cedex 9, France}
\email{dinhvan.duong@math.univ-toulouse.fr}

\begin{document}
	
	\begin{abstract}
		We consider the $L^2$-critical NLS with inverse-square potential
		\[
		i \partial_t u +\Delta u + c|x|^{-2} u = -|u|^{\frac{4}{d}} u, \quad u(0) = u_0, \quad (t,x) \in \R^+ \times \R^d,
		\]
		where $d\geq 3$ and $c\ne 0$ satisfies $c<\lambda(d) := \left(\frac{d-2}{2}\right)^2$. We extend the mass concentration of finite time blow-up solutions established by the first author in \cite{Bensouilah} to $c<\lambda(d)$. Using the profile decomposition, we give a short and simple proof of a limiting profile theorem that yields the same characterization of finite time blow-up solutions with minimal mass obtained by Csobo-Genoud in \cite{CsoboGenoud}. We also extend the characterization obtained by Csobo-Genoud to $c<\lambda(d)$.
	\end{abstract}
	
	\maketitle
	
	\section{Introduction}
	\setcounter{equation}{0}
	
	Consider the Cauchy problem for the focusing $L^2$-critical nonlinear Schr\"odinger equation with inverse-square potential
	\begin{align}
	\left\{
	\begin{array}{rcl}
	i\partial_t u + \Delta u + c|x|^{-2} u &=& - |u|^{\frac{4}{d}} u, \quad (t,x)\in \R^+ \times \R^d, \\
	u(0)&=& u_0,
	\end{array} 
	\right. \label{NLS inverse square}
	\end{align}
	where $d\geq 3$, $u: \R^+ \times \R^d \rightarrow \C$, $u_0:\R^d \rightarrow \C$ and $c\ne 0$ satisfies $c<\lambda(d):=\left(\frac{d-2}{2}\right)^2$. The Schr\"odinger equation with inverse-square potential appears in a variety of physical settings, such as in quantum field equations or black hole solutions of the Einstein's equations (see e.g. \cite{CamblongEpeleFanchiottiCanal} or \cite{KalfSchminckeWalterWust}). The study of nonlinear Schr\"odinger equation with inverse-square potential and power-type nonlinearity has attracted a lot of interests in the last several years (see e.g. \cite{KalfSchminckeWalterWust, BurqPlanchonStalkerTahvildar-Zadeh, OkazawaSuzukiYokota, ZhangZheng, TrachanasZographopoulos, KillipMiaoVisanZhangZheng-sobolev, KillipMiaoVisanZhangZheng-energy, KillipMurphyVisanZheng, LuMiaoMurphy, CsoboGenoud, Dinh-inverse, Bensouilah} and references therein). 
	
	Denote $P_c$ the self-adjoint extension of $-\Delta - c|x|^{-2}$. It is known (see e.g. \cite{KalfSchminckeWalterWust}) that in the range $\lambda(d) -1 <c <\lambda(d)$, the extension is not unique. In this case, we do make a choice among possible extensions, such as Friedrichs extension. The restriction $c<\lambda(d)$ comes from the sharp Hardy inequality
	\begin{align}
	\lambda(d) \int |x|^{-2} |f(x)|^2 dx \leq \int |\nabla f(x)|^2 dx, \quad \forall f \in H^1, \label{sharp hardy inequality}
	\end{align}
	which ensures that $P_c$ is a positive operator. We define the homogeneous Sobolev space $\dot{H}^1_c$ as a completion of $C^\infty_0(\R^d \backslash \{0\})$ under the norm
	\begin{align}
	\|f\|_{\dot{H}^1_c} := \|\sqrt{P_c} f\|_{L^2} = \left( \int |\nabla f(x)|^2 - c|x|^{-2} |f(x)|^2 dx \right)^{1/2}. \label{define dot H1c norm}
	\end{align}
	The sharp Hardy inequality implies that for $c<\lambda(d)$, $\|f\|_{\dot{H}^1_c} \sim \|f\|_{\dot{H}^1}$, and the homogeneous Sobolev space $\dot{H}^1_c$ is equivalent to the usual homogenous Sobolev space $\dot{H}^1$.  
	
	The local well-posedness for $(\ref{NLS inverse square})$ was established by Okazawa-Suzuki-Yokota \cite{OkazawaSuzukiYokota}. 
	\begin{theorem}[Local well-posedness \cite{OkazawaSuzukiYokota}] \label{theorem local well-posedness}
		Let $d\geq 3$ and $c\ne 0$ be such that $c<\lambda(d)$. Then for any $u_0 \in H^1$, there exists $T \in (0, +\infty]$ and a maximal solution $u \in C([0,T), H^1)$ of $(\ref{NLS inverse square})$. The maximal time of existence satisfies either $T=+\infty$ or $T<+\infty$ and
		$
		\lim_{t\uparrow T} \|\nabla u(t)\|_{L^2} =\infty
		$.
		Moreover, the local solution enjoys the conservation of mass and energy
		\begin{align*}
		M(u(t)) &= \int |u(t,x)|^2 dx = M(u_0), \\
		E(u(t)) &= \frac{1}{2} \int |\nabla u(t,x)|^2 dx - \frac{c}{2} \int |x|^{-2} |u(t,x)|^2 dx - \frac{d}{2d+4} \int |u(t,x)|^{\frac{4}{d}+2} dx,
		\end{align*}
		for any $t\in [0,T)$.
	\end{theorem}
	We refer the reader to \cite[Theorem 5.1]{OkazawaSuzukiYokota} for the proof of the above local well-posedness result. Note that the existence of local solutions is based on a refined energy method, and the uniqueness follows from Strichartz estimates which are shown by Burq-Planchon-Stalker-Zadeh in \cite{BurqPlanchonStalkerTahvildar-Zadeh}.
	
	The main purpose of this paper is to study dynamical properties of blow-up solutions to \eqref{NLS inverse square}, including mass concentration, limiting profile and the characterization of finite time blow-up solutions with minimal mass. Such phenomena were extensively studied in the last decades especially for the mass-critical nonlinear Schr\"odinger equation (NLS) (i.e. $c=0$ in $(\ref{NLS inverse square})$). For the mass-critical NLS, the mass concentration was first established by Tsutsumi \cite{Tsutsumi} and Merle-Tsutsumi \cite{MerleTsutsumi}. The limiting profile of finite time blow-up solutions was obtained by Weinstein in \cite{Weinstein-2}. The characterization of finite time blow-up solutions with minimal mass was obtained by Merle in \cite{Merle}. Based on a refined compactness lemma, Hmidi-Keraani in \cite{HmidiKeraani} gave much simpler proofs of all the aforementioned results. It is their approach that we are going to pursue in the sequel.
	
	Following the idea of Hmidi-Keraani in \cite{HmidiKeraani}, to study dynamical properties of finite time blow-up solutions for $(\ref{NLS inverse square})$, we first need the profile decomposition of bounded sequences in $H^1$ related to $(\ref{NLS inverse square})$. This profile decomposition was proved recently by the first author in \cite{Bensouilah}. Thanks to this profile decomposition, a refined version of compactness lemma related to $(\ref{NLS inverse square})$ was shown. With the help of this refined compactness lemma, we are able to study dynamical properties of finite time blow-up solutions for $(\ref{NLS inverse square})$. 
	
	The mass concentration for non-radial blow-up solutions was established by the first author in \cite{Bensouilah} for the case $0<c<\lambda(d)$. Here we extend this result to $c<\lambda(d)$. We also give an improvement of the mass concentration for radial blow-up solutions in the case $c<0$. This improvement is due to the sharp radial Gagliardo-Nirenberg inequality related to $(\ref{NLS inverse square})$ for $c<0$. More precisely, we prove the following result.
	\begin{theorem}[Mass concentration] \label{theorem mass concentration}
		Let $d\geq 3$, $c \ne 0$ and $c<\lambda(d)$. Let $u_0 \in H^1$ be such that the corresponding solution $u$ to $(\ref{NLS inverse square})$ blows up at finite time $0<T<+\infty$. Let $a(t)>0$ be such that 
		\begin{align}
		a(t) \|\nabla u(t)\|_{L^2} \rightarrow \infty, \label{mass concentration condition}
		\end{align}
		as $t\uparrow T$. Then there exists $x(t) \in \R^d$ such that
		\begin{align}
		\liminf_{t\uparrow T} \int_{|x-x(t)| \leq a(t)} |u(t,x)|^2 dx \geq \|Q_{\overline{c}}\|^2_{L^2}, \label{mass concentration}
		\end{align}
		where $\overline{c} =\max\{c,0\}$. Moreover, in the case $c<0$, if  we assume in addition that $u_0$ is radial, then $(\ref{mass concentration})$ can be improved to
		\begin{align}
		\liminf_{t\uparrow T} \int_{|x| \leq a(t)} |u(t,x)|^2 dx \geq \|Q_{c,\emph{rad}}\|^2_{L^2}. \label{radial mass concentration}
		\end{align}
		Here $Q_c$ and $Q_{c,\emph{rad}}$ are given in Theorem $\ref{theorem sharp gagliardo-nirenberg}$.
	\end{theorem}
	\begin{remark}
		\begin{itemize} 
			\item By using a standard argument of Merle-Rapha\"el \cite{MerleRaphael}, we have the following blow-up rate: if $u$ is a solution to $(\ref{NLS inverse square})$ blows up at finite time $0<T<+\infty$, then there exists $C>0$ such that
			\[
			\|\nabla u(t)\|_{L^2} > \frac{C}{\sqrt{T-t}}.
			\]
			\item Rewriting 
			\[
			\frac{1}{a(t)\|\nabla u(t)\|_{L^2}} = \frac{\sqrt{T-t}}{a(t)} \frac{1}{\sqrt{T-t} \|\nabla u(t)\|_{L^2} } < C\frac{\sqrt{T-t}}{a(t)},
			\]
			we see that any function $a(t)>0$ satisfying $\frac{\sqrt{T-t}}{a(t)} \rightarrow 0$ as $t\rightarrow T$ fulfills $(\ref{mass concentration condition})$.
		\end{itemize}
	\end{remark}
	
	The characterization of finite time blow-up solutions for $(\ref{NLS inverse square})$ with minimal mass was recently established by Csobo-Genoud in \cite{CsoboGenoud} in the case $0<c<\lambda(d)$. They showed that up to symmetries of the equation, the only finite time blow-up solutions for $(\ref{NLS inverse square})$ with minimal mass are the pseudo-conformal transformation of ground state standing waves. Note that since the uniqueness of ground states for $(\ref{NLS inverse square})$ is not yet known, one needs to define properly a notion of ground states for $(\ref{NLS inverse square})$. The proof of their result is based on the concentration-compactness lemma (see e.g. \cite[Proposition 1.7.6]{Cazenave}). The key point is the limiting profile result (see \cite[Proposition 4, p.120]{CsoboGenoud}). In this paper, we aim to give a simple proof for the above result of Csobo-Genoud in the case $0<c<\lambda(d)$. Our approach is based on the profile decomposition of \cite{Bensouilah}. This allows us to give a simple version of the limiting profile compared to the one of \cite{CsoboGenoud}. We also extend Csobo-Genoud's result to negative values of $c$. Since the sharp non-radial Gagliardo-Nirenberg inequality for $c<0$ is never attained for $c<0$. We need to restrict our attention only to finite time radial blow-up solutions. More precisely, we prove the following result. 
	\begin{theorem}[Characterization of finite time blow-up solutions with minimal mass] \label{theorem characterization minimal mass}
		\begin{itemize}
			\item Let $d\geq 3$ and $0<c<\lambda(d)$. Let $u_0 \in H^1$ be such that $\|u_0\|_{L^2} = M_{\emph{gs}}$. Suppose that the corresponding solution $u$ to $(\ref{NLS inverse square})$ blows up at finite time $0<T<+\infty$. Then there exist $Q \in \mathcal{G}$, $\theta \in \R$ and $\lambda>0$ such that
			\begin{align}
			u_0(x) = e^{i\theta} e^{i\frac{\lambda^2}{T}} e^{-i\frac{|x|^2}{4T}} \left(\frac{\lambda}{T}\right)^{\frac{d}{2}} Q\left(\frac{\lambda x}{T}\right). \label{characterization initial data}
			\end{align}
			In particular, 
			$
			u(t,x) = S_{Q, T, \theta, \lambda} (t,x)
			$,
			where
			\[
			S_{Q,T,\theta, \lambda}(t,x):= e^{i\theta} e^{i\frac{\lambda^2}{T-t}} e^{-i\frac{|x|^2}{4(T-t)}} \left(\frac{\lambda}{T-t} \right)^{\frac{d}{2}} Q \left( \frac{\lambda x}{T-t}\right).
			\]
			\item Let $d\geq 3$ and $c<0$. Let $u_0 \in H^1_{\emph{rad}}$ be such that $\|u_0\|_{L^2} = M_{\emph{gs,rad}}$. Suppose that the corresponding solution $u$ to $(\ref{NLS inverse square})$ blows up at finite time $0<T<+\infty$. Then there exist $Q_{\emph{rad}} \in \mathcal{G}_{\emph{rad}}$, $\vartheta \in \R$ and $\rho>0$ such that
			\begin{align}
			u_0(x) = e^{i\vartheta} e^{i\frac{\rho^2}{T}} e^{-i\frac{|x|^2}{4T}} \left(\frac{\rho}{T}\right)^{\frac{d}{2}} Q_{\emph{rad}}\left(\frac{\rho x}{T}\right). \label{characterization initial data radial}
			\end{align}
			In particular, 
			$
			u(t,x) = S_{Q_{\emph{rad}}, T, \vartheta, \rho} (t,x)
			$,
			where 
			\[
			S_{Q_{\emph{rad}},T,\vartheta, \rho}(t,x):= e^{i\vartheta} e^{i\frac{\rho^2}{T-t}} e^{-i\frac{|x|^2}{4(T-t)}} \left(\frac{\rho}{T-t} \right)^{\frac{d}{2}} Q_{\emph{rad}} \left( \frac{\rho x}{T-t}\right).
			\]
		\end{itemize}
	\end{theorem}
	We refer the reader to Section $\ref{section blowup profile}$ for the notations $M_{\text{gs}}$, $\mathcal{G}$, $M_{\text{gs,rad}}$ and $\mathcal{G}_{\text{rad}}$. 
	
	The paper is organized as follows. In Section 2, we recall sharp Gagliardo-Nirenberg inequalities and the compactness lemma related to $(\ref{NLS inverse square})$. In Section 3, we give the proof of the mass concentration given in Theorem $\ref{theorem mass concentration}$. In Section 4, we prove a simple version of the limiting profile result compared to the one in \cite{CsoboGenoud}. Using this limiting profile, we give the proof of the characterization of finite time blow-up solutions with minimal mass given in Theorem $\ref{theorem characterization minimal mass}$. 
	
	\section{Preliminaries}
	\setcounter{equation}{0}
	
	\subsection{Sharp Gagliardo-Nirenberg inequalities}
	In this subsection, we recall sharp Gagliardo-Nirenberg inequalities related to $(\ref{NLS inverse square})$. Let us start with the sharp non-radial Gagliardo-Nirenberg inequality
	\begin{align}
	\|u\|^{\frac{4}{d} +2}_{L^{\frac{4}{d}+2}} \leq C_{\text{GN}}(c) \|u\|^{\frac{4}{d}}_{L^2} \|u\|^2_{\dot{H}^1_c}, \label{sharp gagliardo-nirenberg}
	\end{align}
	where the sharp constant $C_{\text{GN}}(c)$ is defined by
	\[
	C_{\text{GN}}(c):= \sup \left\{ J_c(u) \ : \ u \in H^1 \backslash \{0\} \right\}.
	\]
	Here $J_c(u)$ is the Weinstein functional
	\begin{align}
	J_c(u):= \|u\|^{\frac{4}{d}+2}_{L^{\frac{4}{d}+2}} \div \left[ \|u\|^{\frac{4}{d}}_{L^2} \|u\|^2_{\dot{H}^1_c} \right]. \label{define weinstein functional}
	\end{align}
	We also recall the sharp radial Gagliardo-Nirenberg inequality
	\begin{align}
	\|u\|^{\frac{4}{d}+2}_{L^{\frac{4}{d}+2}} \leq C_{\text{GN}}(c,\text{rad}) \|u\|^{\frac{4}{d}}_{L^2} \|u\|^2_{\dot{H}^1_c},  \label{sharp radial gagliardo-nirenberg}
	\end{align}
	where the sharp constant $C_{\text{GN}}(c,\text{rad})$ is defined by
	\[
	C_{\text{GN}}(c,\text{rad}) : = \sup \left\{ J_c(u) \ : \ u \in H^{1}_{\text{rad}} \backslash \{0\}\right\},
	\]
	where $H^1_{\text{rad}}$ is the space of radial $H^1$ functions. When $c=0$, Weinstein in \cite{Weinstein} proved that the sharp constant $C_{\text{GN}}(0)$ is attained by the fuction $Q_0$ which is the unique (up to symmetries) positive radial solution of
	\begin{align}
	\Delta Q_0 -Q_0 + |Q_0|^{\frac{4}{d}} Q_0=0. \label{elliptic equation c=0}
	\end{align}
	We have the following result (see \cite{KillipMurphyVisanZheng} and also \cite{Dinh-inverse}).
	\begin{theorem}[Sharp Gagliardo-Nirenberg inequalities] \label{theorem sharp gagliardo-nirenberg}
		Let $d\geq 3$ and $c\ne 0$ be such that $c<\lambda(d)$. Then $C_{\emph{GN}}(c) \in (0,\infty)$ and
		\begin{itemize}
			\item if $0<c<\lambda(d)$, then the equality in $(\ref{sharp gagliardo-nirenberg})$ is attained by a function $Q_c \in H^1$ which is a positive radial solution to the elliptic equation
			\begin{align}
			\Delta Q_c + c|x|^{-2} Q_c - Q_c + |Q_c|^{\frac{4}{d}} Q_c =0. \label{elliptic equation}
			\end{align}
			\item if $c<0$, then $C_{\emph{GN}}(c) = C_{\emph{GN}}(0)$ and the equality in $(\ref{sharp gagliardo-nirenberg})$ is never attained. However, the equality in $(\ref{sharp radial gagliardo-nirenberg})$ is attained by a function $Q_{c,\emph{rad}} \in H^1_{\emph{rad}}$ which is a positive solution to the elliptic equation
			\begin{align}
			\Delta Q_{c,\emph{rad}} + c|x|^{-2} Q_{c,\emph{rad}} - Q_{c,\emph{rad}} + |Q_{c,\emph{rad}}|^{\frac{4}{d}} Q_{c,\emph{rad}} =0. \label{radial elliptic equation}
			\end{align}
		\end{itemize}
	\end{theorem}
	
	We refer the reader to \cite[Theorem 3.1]{KillipMurphyVisanZheng} (see also \cite[Theorem 4.1]{Dinh-inverse}) for the proof of the above result.
	
	\begin{remark} 
		\begin{itemize}
			\item In the case $0<c<\lambda(d)$, Theorem $\ref{theorem sharp gagliardo-nirenberg}$ shows that there exist positive radial solutions to the elliptic equation $(\ref{elliptic equation})$. However, unlike the case $c=0$, the uniqueness up to symmetries of these solutions is not known yet. We also have the following Pohozaev's identities \footnote{These identities can be proved rigorously by using the technique of \cite[Proposition 1]{BerestyckiLions}: first, considering Pohozaev's identities in $\Omega_{r,R}:= \{x \ : \ r <|x| <R\}$, and then showing the boundary term (on $\partial \Omega_{r,R}$) to converge to $0$ as $r \rightarrow 0$ and $R \rightarrow +\infty$.}:
			\begin{align*}
			\|Q_c\|^2_{L^2} = \frac{2}{d}\|Q_c\|^2_{\dot{H}^1_c} = \frac{2}{d+2} \|Q_c\|^{\frac{4}{d}+2}_{L^{\frac{4}{d}+2}}. 
			\end{align*}
			In particular,
			\[
			C_{\text{GN}}(c) = \frac{d+2}{d} \frac{1}{\|Q_c\|^{\frac{4}{d}}_{L^2}}.
			\]
			\item Since the above identities still hold true for $c=0$, we get from Theorem $\ref{theorem sharp gagliardo-nirenberg}$ that for any $c<\lambda(d)$,
			\begin{align}
			C_{\text{GN}}(c) = \frac{d+2}{d} \frac{1}{\|Q_{\overline{c}}\|^{\frac{4}{d}}_{L^2}}, \label{sharp constant gagliardo-nirenberg}
			\end{align}
			where $\overline{c}=\max\{c,0\}$.
			\item In the case $c<0$, we also have
			\begin{align*}
			\|Q_{c,\text{rad}}\|^2_{L^2} = \frac{2}{d}\|Q_{c,\text{rad}}\|^2_{\dot{H}^1_c} = \frac{2}{d+2} \|Q_{c,\text{rad}}\|^{\frac{4}{d}+2}_{L^{\frac{4}{d}+2}}. 
			\end{align*}
			In particular,
			\begin{align}
			C_{\text{GN}}(c,\text{rad}) = \frac{d+2}{d} \frac{1}{\|Q_{c,\text{rad}}\|^{\frac{4}{d}}_{L^2}}. \label{sharp constant gagliardo-nirenberg radial}
			\end{align}
			Note that since $C_{\text{GN}}(c,\text{rad})<C_{\text{GN}}(c)$, we see that for any $c<0$, 
			\[
			\|Q_0\|_{L^2} < \|Q_{c,\text{rad}}\|_{L^2}.
			\]
		\end{itemize}
	\end{remark}
	
	\subsection{Profile decomposition}
	In this subsection, we recall the profile decomposition related to the nonlinear Schr\"odinger equation with inverse-square potential. This profile decomposition was established recently by the first author in \cite{Bensouilah} for $0<c<\lambda(d)$. There is no difficulty to extend this result for negative values of $c$.
	\begin{proposition} [Profile decomposition] \label{proposition profile decomposition}
		Let $d\geq 3$ and $c <\lambda(d)$. Let $(v_n)_{n\geq 1}$ be a bounded sequence in $H^1$. Then there exist a subsequence still denoted by $(v_n)_{n\geq 1}$, a family $(x_n^j)_{n\geq 1}$ of sequences in $\R^d$ and a sequence $(V^j)_{j\geq 1}$ of $H^1$-functions such that
		\begin{itemize}
			\item[i)] for every $j\ne k$, 
			\begin{align}
			|x^j_n-x^k_n| \rightarrow \infty, \label{pairwise orthogonality}
			\end{align}
			as $n\rightarrow \infty$;
			\item[ii)] for every $l \geq 1$ and every $x \in \R^d$, we have
			\[
			v_n(x) = \sum_{j=1}^l V^j(x-x^j_n) + v^l_n(x),
			\]
			with 
			\begin{align}
			\limsup_{n\rightarrow \infty} \|v^l_n\|_{L^q} \rightarrow 0, \label{profile error}
			\end{align}
			as $l \rightarrow \infty$ for every $2<q<\frac{2d}{d-2}$.
		\end{itemize}
		Moreover, for every $l\geq 1$,
		\begin{align}
		\|v_n\|^2_{L^2} &= \sum_{j=1}^l \|V^j\|^2_{L^2} + \|v^l_n\|^2_{L^2} + o_n(1), \label{L2 norm expansion}\\
		\|v_n\|^2_{\dot{H}^1_c} &= \sum_{j=1}^l \|V^j(\cdot- x^j_n)\|_{\dot{H}^1_c} + \|v^l_n\|_{\dot{H}^1_c} + o_n(1), \label{H1c norm expansion}
		\end{align}
		as $n\rightarrow \infty$.
	\end{proposition}
	
	\begin{proof}
		For reader's convenience, we recall some details. Since $H^1$ is a Hilbert space, we denote $\Omega(v_n)$ the set of functions obtained as weak limits of sequences of the translated $v_n(\cdot + x_n)$ with $(x_n)_{n\geq 1}$ a sequence in $\R^d$. Denote
		\[
		\eta(v_n):= \sup \{ \|v\|_{L^2} + \|\nabla v\|_{L^2} : v \in \Omega(v_n)\}.
		\]
		Clearly,
		\[
		\eta(v_n) \leq \limsup_{n\rightarrow \infty} \|v_n\|_{L^2} + \|\nabla v_n\|_{L^2}.
		\]
		We shall prove that there exist a sequence $(V^j)_{j\geq 1}$ of $\Omega(v_n)$ and a family $(x_n^j)_{j\geq 1}$ of sequences in $\R^d$ such that for every $k \ne j$,
		\[
		|x_n^k - x_n^j| \rightarrow \infty, \quad \text{as } n \rightarrow \infty,
		\]
		and up to a subsequence, the sequence $(v_n)_{n\geq 1}$ can be written, for every $l\geq 1$ and every $x \in \R^d$, as
		\[
		v_n(x) = \sum_{j=1}^l V^j(x-x_n^j) + v^l_n(x), 
		\]
		with $\eta(v^l_n) \rightarrow 0$ as $l \rightarrow \infty$. Moreover, the identities $(\ref{L2 norm expansion})$ and $(\ref{H1c norm expansion})$ hold as $n \rightarrow \infty$. \newline
		\indent Indeed, if $\eta(v_n) =0$, then we can take $V^j=0$ for all $j\geq 1$. Otherwise we choose $V^1 \in \Omega(v_n)$ such that
		\[
		\|V^1\|_{L^2} + \|\nabla V^1\|_{L^2} \geq \frac{1}{2} \eta(v_n) >0. 
		\]
		By the definition of $\Omega(v_n)$, there exists  a sequence $(x^1_n)_{n\geq 1} \subset \R^d$ such that up to a subsequence,
		\[
		v_n(\cdot + x^1_n) \rightharpoonup V^1 \text{ weakly in } H^1.
		\]
		Set $v_n^1(x):= v_n(x) - V^1(x-x^1_n)$. We see that $v^1_n(\cdot + x^1_n) \rightharpoonup 0$ weakly in $H^1$ and
		thus
		\begin{align*}
		\|v_n\|^2_{L^2} &= \|V^1\|^2_{L^2} + \|v^1_n\|^2_{L^2} + o_n(1),  \\
		\|\nabla v_n\|^2_{L^2} &= \|\nabla V^1\|^2_{L^2} + \|\nabla v^1_n\|^2_{L^2} + o_n(1),
		\end{align*}
		as $n \rightarrow \infty$. We next show that 
		\[
		\int |x|^{-2} |v_n(x)|^2 dx = \int |x|^{-2} |V^1(x-x^1_n)|^2 dx + \int |x|^{-2} |v^1_n(x)|^2 dx + o_n(1),
		\]
		as $n \rightarrow \infty$. Using the fact 
		\[
		|v_n(x)|^2 = |V^1(x-x^1_n)|^2 + |v^1_n(x)|^2 + 2\re{ (V^1(x-x^1_n) \overline{v}^1_n(x))},
		\]
		it suffices to show that
		\begin{align}
		\int |x|^{-2} V^1(x-x^1_n) \overline{v}^1_n(x) dx \rightarrow 0, \label{H1c norm estimate proof}
		\end{align}
		as $n\rightarrow \infty$. Without loss of generality, we may assume that $V^1$ is continuous and compactly supported. Moreover, up to a subsequence, we assume that $|x^1_n| \rightarrow \{0,\infty\}$ as $n\rightarrow \infty$. 
		
		$\bullet$ Case 1: $|x^1_n| \rightarrow \infty$. Since $|x^1_n| \rightarrow \infty$ as $n\rightarrow \infty$, we see that $|x+ x^1_n| \geq 1$ for all $x \in \text{supp}(V^1)$ and all $n \geq n_0$ with $n_0$ large enough. Therefore, for $n\geq n_0$,
		\begin{align*}
		\left|\int |x|^{-2} V^1(x-x^1_n) \overline{v}^1_n(x) dx\right| &= \int_{\text{supp}(V^1)} |x+x^1_n|^{-2} |V^1(x)| |v^1_n(x+x^1_n)| dx \\
		&\leq \int |V^1(x)| |v^1_n(x+ x^1_n)| dx.
		\end{align*}
		Since $v^1_n(\cdot+ x^1_n) \rightharpoonup 0$ in $H^1$ as $n\rightarrow \infty$, the last term tends to zero as $n\rightarrow \infty$.
		
		$\bullet$ Case 2: $|x^1_n| \rightarrow 0$. Let $\eps>0$. For $\eta>0$ small to be chosen later, we split
		\begin{align}
		\int_{\text{supp}(V^1)} |x+x^1_n|^{-2} |V^1(x)| |v^1_n(x+x^1_n)| dx &= \int_{B(0,\eta)} |x+x^1_n|^{-2} |V^1(x)| |v^1_n(x+x^1_n)| dx \label{spliting}\\
		&\mathrel{\phantom{=}}+ \int_{\text{supp}(V^1)\backslash B(0,\eta)} |x+x^1_n|^{-2} |V^1(x)| |v^1_n(x+x^1_n)| dx. \nonumber
		\end{align}
		Since $|x^1_n| \rightarrow 0$, we see that for all $n\geq n_1$ with $n_1$ large enough, $|x+x_n^1| \geq \eta/2$ for all $x \in \text{supp}(V^1) \backslash B(0,\eta)$. Thus
		\begin{align*}
		\int_{\text{supp}(V^1)\backslash B(0,\eta)} |x+x^1_n|^{-2} |V^1(x)| |v^1_n(x+x^1_n)| dx \lesssim \eta^{-2} \int |V^1(x)| |v^1_n(x+x^1_n)| dx. 
		\end{align*}
		We next learn from the fact $v^1_n(\cdot+x^1_n) \rightharpoonup 0$ in $H^1$ as $n\rightarrow \infty$ that for $n\geq n_1$ (increasing $n_1$ if necessary),
		\begin{align}
		\int_{\text{supp}(V^1)\backslash B(0,\eta)} |x+x^1_n|^{-2} |V^1(x)| |v^1_n(x+x^1_n)| dx < \frac{\eps}{2}. \label{term 2}
		\end{align}
		We next use the Cauchy-Schwarz inequality, Hardy's inequality $(\ref{sharp hardy inequality})$ and the fact $(v_n^1)_{n\geq 1}$ is bounded in $H^1$ to get
		\begin{align}
		\int_{B(0,\eta)} |x+x^1_n|^{-2} |V^1(x)| |v^1_n(x+x^1_n)| dx &\leq \Big( \int_{B(0,\eta)} |x+x^1_n|^{-2} |V^1(x)|^2 dx \Big)^{1/2} \nonumber \\
		&\mathrel{\phantom{\leq \Big( \int_{B(0,\eta)}}} \times \Big( \int |x+x^1_n|^{-2} |v^1_n(x+x^1_n)|^2 dx \Big)^{1/2}  \nonumber \\
		&\lesssim \Big( \int_{B(0,\eta)} |x+x^1_n|^{-2} |V^1(x)|^2 dx \Big)^{1/2}  \|\nabla v^1_n\|_{L^2} \nonumber \\
		&\lesssim \Big( \int_{B(0,\eta)} |x+x^1_n|^{-2} |V^1(x)|^2 dx \Big)^{1/2}. \label{spliting term 1}
		\end{align}
		Since $|V^1(x)|^2$ is continuous on the compact set $\overline{B}(0,3\eta)$, hence it is uniformly continuous on $\overline{B}(0,3\eta)$. Thus, there exists $\delta>0$ such that for all $x,y \in \overline{B}(0,3\eta)$ satisfying $|x-y| <\delta$, we have
		\[
		||V^1(x)|^2- |V^1(y)|^2| <\frac{\eps^2}{8K(\eta)},
		\]
		where 
		\[
		K(\eta):= \int_{B(0,2\eta)} |x|^{-2} dx = \frac{(2\eta)^{d-2}}{d-2} |\mathbb{S}^{d-1}|.
		\]
		Note that we can take $\delta \in (0, \eta)$. Since $|x^1_n| \rightarrow 0$, we have for $n\geq n_2$ with $n_2$ large enough that
		\[
		|x^1_n| <\delta<\eta.
		\]
		This implies that for all $x \in B(0,2\eta)$ and all $n\geq n_2$,
		\begin{align}
		||V^1(x-x^1_n)|^2- |V^1(x)|^2| <\frac{\eps^2}{8K(\eta)}. \label{estimate V}
		\end{align}
		Since $B(x^1_n,\eta) \subset B(0,2\eta)$ for all $n\geq n_2$, we use $(\ref{estimate V})$ to get
		\begin{align*}
		\int_{B(0,\eta)} |x+x^1_n|^{-2} |V^1(x)|^2 dx &=\int_{B(x^1_n,\eta)} |x|^{-2} |V^1(x-x^1_n)|^2 dx \\
		&\leq \int_{B(0,2\eta)} |x|^{-2} |V^1(x)|^2 dx + \int_{B(0,2\eta)} |x|^{-2} \frac{\eps^2}{8 K(\eta)} dx \\
		&= \int_{B(0,2\eta)} |x|^{-2} |V^1(x)|^2 dx + \frac{\eps^2}{8}. 
		\end{align*}
		Using Hardy's inequality $(\ref{sharp hardy inequality})$ with $V^1\in H^1$, the dominated convergence allows to choose $\eta>0$ small enough so that
		\[
		\int_{B(0,2\eta)} |x|^{-2} |V^1(x)|^2 dx <\frac{\eps^2}{8}.
		\]
		We thus obtain
		\[
		\int_{B(0,\eta)} |x+x^1_n|^{-2} |V^1(x)|^2 dx <\frac{\eps^2}{4}, 
		\]
		which together with $(\ref{spliting term 1})$ yield for $n\geq n_2$,
		\begin{align}
		\int_{B(0,\eta)} |x+x^1_n|^{-2} |V^1(x)| |v^1_n(x+x^1_n)| dx  <\frac{\eps}{2}. \label{term 1}
		\end{align}
		Combining $(\ref{spliting})$, $(\ref{term 2})$ and $(\ref{term 1})$, we have for $n \geq \max\{n_1,n_2\}$,
		\[
		\int_{\text{supp}(V^1)} |x+x^1_n|^{-2} |V^1(x)| |v^1_n(x+x^1_n)| dx <\eps.
		\]
		Therefore, $(\ref{H1c norm estimate proof})$ is proved in both cases.
		
		We now replace $(v_n)_{n\geq 1}$ by $(v^1_n)_{n\geq 1}$ and repeat the same process. If $\eta(v^1_n) =0$, then we choose $V^j=0$ for all $j \geq 2$. Otherwise there exist $V^2 \in \Omega(v^1_n)$ and a sequence $(x^2_n)_{n\geq 1} \subset \R^d$ such that
		\[
		\|V^2\|_{L^2} + \|\nabla V^2\|_{L^2} \geq \frac{1}{2} \eta(v^1_n)>0,
		\]
		and
		\[
		v^1_n(\cdot+x^2_n) \rightharpoonup V^2 \text{ weakly in } H^1.
		\]
		Set $v^2_n(x) := v^1_n(x) - V^2(x-x^2_n)$. We thus have
		$v^2_n(\cdot +x^2_n) \rightharpoonup 0$ weakly in $H^1$ and 
		\begin{align*}
		\|v^1_n\|^2_{L^2} & = \|V^2\|^2_{L^2} + \|v^2_n\|^2_{L^2} + o_n(1), \\
		\|v^1_n\|^2_{\dot{H}^1_c} &= \| V^2\|^2_{\dot{H}^1_c} + \|v^2_n\|^2_{\dot{H}^1_c} + o_n(1),
		\end{align*}
		as $n \rightarrow \infty$. We claim that 
		\[
		|x^1_n - x^2_n| \rightarrow \infty, \quad \text{as } n \rightarrow \infty.
		\]
		In fact, if it is not true, then up to a subsequence, $x^1_n - x^2_n \rightarrow x_0$ as $n \rightarrow \infty$ for some $x_0 \in \R^d$. Since 
		\[
		v^1_n(x + x^2_n) = v^1_n(x +(x^2_n -x^1_n) + x^1_n),
		\]
		and $v^1_n (\cdot + x^1_n)$ converges weakly to $0$, we see that $V^2=0$. This implies that $\eta(v^1_n)=0$ and it is a contradiction. An argument of iteration and orthogonal extraction allows us to construct the family $(x^j_n)_{j\geq 1}$ of sequences in $\R^d$ and the sequence $(V^j)_{j\geq 1}$ of $H^1$ functions satisfying the claim above. Furthermore, the convergence of the series $\sum_{j\geq 1}^\infty \|V^j\|^2_{L^2} + \|\nabla V^j\|^2_{L^2}$ implies that 
		\[
		\|V^j\|^2_{L^2} + \|\nabla V^j\|^2_{L^2} \rightarrow 0, \quad \text{as } j \rightarrow \infty.
		\]
		By construction, we have
		\[
		\eta(v^j_n) \leq 2 \left(\|V^{j+1}\|_{L^2} + \|\nabla V^{j+1}\|_{L^2}\right),
		\]
		which proves that $\eta(v^j_n) \rightarrow 0$ as $j \rightarrow \infty$. The proof of $\ref{proposition profile decomposition}$ follows by the same lines as in \cite[Proposition 2.3]{HmidiKeraani}. We thus omit the details.
	\end{proof}
	
	\subsection{Compactness lemma}
	In this subsection, we recall a compactness lemma related to the nonlinear Schr\"odinger equation with inverse-square potential. 
	\begin{lemma}[Compactness lemma] \label{lemma compactness lemma}
		Let $d\geq 3$, $c \ne 0$ and $c <\lambda(d)$. Let $(v_n)_{n\geq 1}$ be a bounded sequence in $H^1$ such that
		\begin{align}
		\limsup_{n \rightarrow \infty} \|v_n\|_{\dot{H}^1_c} \leq M, \quad \limsup_{n\rightarrow \infty} \|v_n\|_{L^{\frac{4}{d}+2}} \geq m. \label{compactness lemma condition}
		\end{align}
		Then there exists $(x_n)_{n\geq 1}$ in $\R^d$ such that up to a subsequence, 
		$
		v_n(\cdot +x_n) \rightharpoonup V \text{ weakly in } H^1
		$
		for some $V \in H^1$ satisfying
		\begin{align*}
		\|V\|^{\frac{4}{d}}_{L^2} \geq \frac{d}{d+2} \frac{m^{\frac{4}{d}+2} }{M^2} \|Q_{\overline{c}}\|^{\frac{4}{d}}_{L^2}, 
		\end{align*}
		where $\overline{c} = \max \{c,0\}$ and $Q_c$ is given in Theorem $\ref{theorem sharp gagliardo-nirenberg}$. Moreover, in the case $c<0$, if we assume in addition $(v_n)_{n\geq 1}$ are radially symmetric, then up to a subsequence
		$
		v_n \rightharpoonup V \text{ weakly in } H^1
		$
		for some $V \in H^1_{\emph{rad}}$ satisfying
		\begin{align*}
		\|V\|^{\frac{4}{d}}_{L^2} \geq \frac{d}{d+2} \frac{m^{\frac{4}{d}+2} }{M^2} \|Q_{c, \emph{rad}}\|^{\frac{4}{d}}_{L^2}, 
		\end{align*}
		where $Q_{c,\emph{rad}}$ is also given in Theorem $\ref{theorem sharp gagliardo-nirenberg}$.
	\end{lemma}
	\begin{proof}
		In the case $c<\lambda(d)$ and $v_n$ non-radial, the proof is given in \cite[Lemma 5]{Bensouilah} using the profile decomposition, the sharp Gagliardo-Nirenberg inequality $(\ref{sharp gagliardo-nirenberg})$ and $(\ref{sharp constant gagliardo-nirenberg})$. 
		
		Let us now consider the case $c<0$ and $(v_n)_{n\geq 1}$ a bouded sequence in $H^1_{\text{rad}}$ satisfying $(\ref{compactness lemma condition})$. Thanks to the fact
		\[
		H^1_{\text{rad}} \hookrightarrow L^{\frac{4}{d}+2} \text{ compactly},
		\]
		we see that there exists $V \in H^1_{\text{rad}}$ such that up to a subsequence, $v_n \rightharpoonup V$ weakly in $H^1$ as well as strongly in $L^{\frac{4}{d}+2}$. In particular, we have from the second condition in $(\ref{compactness lemma condition})$ that
		$m \leq \|V\|_{L^{\frac{4}{d}+2}}$. By the sharp radial Gagliardo-Nirenberg inequality and $(\ref{sharp constant gagliardo-nirenberg radial})$, we have
		\[
		m^{\frac{4}{d}+2} \leq \|V\|^{\frac{4}{d}+2}_{L^{\frac{4}{d}+2}} \leq \frac{d+2}{d} \frac{1}{\|Q_{c,\text{rad}}\|_{L^2}^{\frac{4}{d}} } \|V\|^{\frac{4}{d}}_{L^2} \|V\|^2_{\dot{H}^1_c}.
		\]
		By the lower semi continuity of Hardy's functional, the first condition in $(\ref{compactness lemma condition})$ implies
		\[
		\|V\|_{\dot{H}^1_c}  \leq \limsup_{n\rightarrow \infty} \|v_n\|_{\dot{H}^1_c} \leq M.
		\]
		We thus obtain
		\[
		\|V\|_{L^2}^{\frac{4}{d}} \geq \frac{d}{d+2} \frac{m^{\frac{4}{d}+2} }{M^2} \|Q_{c, \text{rad}}\|^{\frac{4}{d}}_{L^2}.
		\]
		The proof is complete.
	\end{proof}
	
	\section{Mass concentration} \label{section mass concentration}
	\setcounter{equation}{0} 
	In this short section, we give the proof of the mass concentration given in Theorem $\ref{theorem mass concentration}$. 
	
	\noindent \textit{Proof of Theorem $\ref{theorem mass concentration}$.}
	The proof is similar to the one of \cite[Theorem 1]{Bensouilah}. For the sake of completeness, we recall some details. Let $(t_n)_{n\geq 1}$ be a time sequence such that $t_n \uparrow T$ as $n\rightarrow \infty$. Set 
	\[
	\lambda_n := \frac{\|Q_{\overline{c}}\|_{\dot{H}^1_c}}{\|u(t_n)\|_{\dot{H}^1_c}}, \quad v_n(x) := \lambda_n^{\frac{d}{2}} u(t_n, \lambda_n x). 
	\]
	By the local well-posedness theory given in Theorem $\ref{theorem local well-posedness}$ and the equivalence between $\dot{H}^1_c$ and $\dot{H}^1$, we see that $\lambda_n \rightarrow 0$ as $n\rightarrow \infty$. Moreover, a direct computation combined with the conservation of mass and energy show
	\[
	\|v_n\|_{L^2} = \|u(t_n)\|_{L^2} = \|u_0\|_{L^2}, \quad \|v_n\|_{\dot{H}^1_c} = \lambda_n \|u(t_n)\|_{\dot{H}^1_c} = \|Q_{\overline{c}}\|_{\dot{H}^1_c},
	\]
	and
	\[
	E(v_n) = \lambda_n^2 E(u(t_n)) = \lambda_n^2 E(u_0) \rightarrow 0,
	\]
	as $n\rightarrow \infty$. In particular,
	\[
	\|v_n\|^{\frac{4}{d}+2}_{L^{\frac{4}{d}+2}} \rightarrow \frac{d+2}{d} \|Q_{\overline{c}}\|^2_{\dot{H}^1_c},
	\]
	as $n\rightarrow \infty$. This implies in particular that $(v_n)_{n\geq 1}$ satisfies conditions of Lemma $\ref{lemma compactness lemma}$ with 
	\[
	m^{\frac{4}{d}+2} = \frac{d+2}{d} \|Q_{\overline{c}}\|^2_{\dot{H}^1_c}, \quad M^2 = \|Q_{\overline{c}}\|^2_{\dot{H}^1_c}.
	\]
	Therefore, there exist a sequence $(x_n)_{n\geq 1}$ in $\R^d$  and $V\in H^1$ such that up to a subsequence
	\[
	v_n(\cdot+x_n) = \lambda_n^{\frac{d}{2}} u(t_n, \lambda_n \cdot + x_n) \rightharpoonup V \text{ weakly in } H^1,
	\]
	as $n\rightarrow \infty$ with $\|V\|_{L^2} \geq \|Q_{\overline{c}}\|_{L^2}$. This implies for every $R>0$, 
	\[
	\liminf_{n\rightarrow \infty} \int_{|x| \leq R} \lambda_n^d |u(t_n, \lambda_n x + x_n)|^2 dx \geq \int_{|x| \leq R} |V(x)|^2 dx,
	\]
	hence
	\[
	\liminf_{n\rightarrow \infty} \int_{|x-x_n| \leq R\lambda_n} |u(t_n, x)|^2 dx \geq \int_{|x| \leq R} |V(x)|^2 dx.
	\]
	Since 
	\[
	a(t_n) \|\nabla u(t_n)\|_{L^2} = \frac{a(t_n)}{\lambda_n} \frac{\|\nabla u(t_n)\|_{L^2}}{\|u(t_n)\|_{\dot{H}^1_c}} \|Q_{\overline{c}}\|_{\dot{H}^1_c},
	\]
	the equivalence $\|\nabla u(t_n)\|_{L^2} \sim \|u(t_n)\|_{\dot{H}^1_c}$ and the condition $(\ref{mass concentration condition})$ yield $\frac{a(t_n)}{\lambda_n} \rightarrow \infty$ as $n\rightarrow \infty$. We thus get for every $R>0$,
	\[
	\liminf_{n\rightarrow \infty} \sup_{y\in \R^d} \int_{|x-y| \leq a(t_n)} |u(t_n,x)|^2 dx \geq \int_{|x| \leq R} |V(x)|^2 dx,
	\]
	which means that
	\[
	\liminf_{n\rightarrow \infty} \sup_{y\in \R^d} \int_{|x-y|\leq a(t_n)} |u(t_n,x)|^2 dx \geq \int |V(x)|^2 dx \geq \int |Q_{\overline{c}}(x)|^2 dx.
	\]
	Since the sequence $(t_n)_{n\geq 1}$ is arbitrary, we infer that
	\[
	\liminf_{t\uparrow T} \sup_{y\in \R^d} \int_{|x-y| \leq a(t)} |u(t,x)|^2 dx \geq \int |Q_{\overline{c}}(x)|^2 dx.
	\]
	Moreover, since for every $t\in (0,T)$, the function $u \mapsto \int_{|x-y| \leq a(t)} |u(t,x)|^2 dx$ is continuous and goes to zero at inifinity, there exists $x(t) \in \R^d$ such that 
	\[
	\sup_{y\in \R^d} \int_{|x-y| \leq a(t)} |u(t,x)|^2 dx = \int_{|x-x(t)| \leq a(t)} |u(t,x)|^2 dx.
	\]
	This completes the first part of Theorem $\ref{theorem mass concentration}$. 
	
	We now consider the case $c<0$ and assume $u_0 \in H^1_{\text{rad}}$. It is well-known that the corresponding solution $u(t)$ to $(\ref{NLS inverse square})$ with initial data $u_0$ is also in $H^1_{\text{rad}}$ for any $t$ in the existence time. Let $(t_n)_{n\geq 1}$ be such that $t_n \uparrow T$ as $n\rightarrow \infty$. Denote
	\[
	\rho_n := \frac{\|Q_{c, \text{rad}}\|_{\dot{H}^1_c}}{\|u(t_n)\|_{\dot{H}^1_c}}, \quad v_n(x) := \rho_n^{\frac{d}{2}} u(t_n, \rho_n x). 
	\]
	As above, the blow-up alternative implies $\rho_n \rightarrow 0$ as $n\rightarrow \infty$. We also have
	\[
	\|v_n\|_{L^2} = \|u_0\|_{L^2}, \quad \|v_n\|_{\dot{H}^1_c} = \rho_n \|u(t_n)\|_{\dot{H}^1_c} = \|Q_{c,\text{rad}}\|_{\dot{H}^1_c},
	\]
	and 
	\[
	E(v_n) = \rho_n^2 E(u(t_n)) = \rho_n^2 E(u_0) \rightarrow 0,
	\]
	as $n\rightarrow \infty$. This implies in particular that
	\[
	\|v_n\|_{L^{\frac{4}{d}+2}}^{\frac{4}{d}+2} \rightarrow \frac{d+2}{d} \|Q_{c,\text{rad}}\|^2_{\dot{H}^1_c},
	\]
	as $n\rightarrow \infty$. We thus obtain a bounded sequence $(v_n)_{n\geq 1}$ in $H^1_{\text{rad}}$ satisfying conditions of Lemma $\ref{lemma compactness lemma}$ with
	\[
	m^{\frac{4}{d}+2} =\frac{d+2}{d} \|Q_{c,\text{rad}}\|^2_{\dot{H}^1_c}, \quad M^2 = \|Q_{c,\text{rad}}\|^2_{\dot{H}^1_c}.
	\]
	Thus, there exists $V\in H^1_{\text{rad}}$ such that 
	\[
	v_n \rightharpoonup V \text{ weakly in } H^1,
	\]
	as $n\rightarrow \infty$ with $\|V\|_{L^2} \geq \|Q_{c,\text{rad}}\|_{L^2}$. The rest of the proof follows by the same argument as in the first case. The proof is complete.
	\defendproof	
	
	\section{Characterization of finite time blow-up solutions with minimal mass} \label{section blowup profile}
	\setcounter{equation}{0} 
	In this section, we give the proof of the characterization of finite time blow-up solutions with minimal mass given in Theorem $\ref{theorem characterization minimal mass}$. Let us start with the following variational structure of ground states. 
	
	\subsection{Variational structure of ground states}
	In this subsection, we show the variational structure of ground states which is neccessary in the study of limiting profile of finite time blow-up solutions with minimal mass. To sucessfully study the variational structure of ground states, we need to define a proper notion of ground states. To do this, we follow the idea of Csobo-Genoud in \cite{CsoboGenoud}.
	\begin{definition} [Ground states] \label{definition ground states}
		\begin{itemize}
			\item In the case $0<c<\lambda(d)$, we call \textbf{ground states} the maximizers of $J_c$ (see $(\ref{define weinstein functional})$) which are positive radial solutions to the elliptic equation $(\ref{elliptic equation})$. The set of ground states is denoted by $\mathcal{G}$. 
			\item In the case $c<0$, we call \textbf{radial ground states} the maximizers of $J_c$ which are positive radial solutions to the elliptic equation $(\ref{radial elliptic equation})$. The set of radial ground states is denoted by $\mathcal{G}_{\text{rad}}$. 
		\end{itemize} 
	\end{definition}
	\begin{remark}
		\begin{itemize}
			\item The reason for introducing the above notion of ground states is that the uniqueness (up to symmetries) of positive radial solutions to $(\ref{elliptic equation})$ and $(\ref{radial elliptic equation})$ are not yet known. 
			
			\item By definition, the function $Q_c$ (resp. $Q_{c,\text{rad}}$) given in Theorem $\ref{theorem sharp gagliardo-nirenberg}$ belongs to $\mathcal{G}$ (resp. $\mathcal{G}_{\text{rad}}$). 
			
			\item It follows from the proof of Theorem $\ref{theorem sharp gagliardo-nirenberg}$ and $(\ref{sharp constant gagliardo-nirenberg})$ that all ground states have the same mass. Hence, there exists $M_{\text{gs}}>0$ such that $\|Q\|_{L^2} = M_{\text{gs}}$ for all $Q \in \mathcal{G}$. The constant $M_{\text{gs}}$ is called \textbf{minimal mass}.
			
			\item Similarly, it follows from the proof of Theorem $\ref{theorem sharp gagliardo-nirenberg}$ and $(\ref{sharp constant gagliardo-nirenberg radial})$ that all radial ground states have the same mass. Hence there exists $M_{\text{gs,rad}}>0$ such that $\|Q_{\text{rad}}\|_{L^2} = M_{\text{gs,rad}}$ for all $Q_{\text{rad}} \in \mathcal{G}_{\text{rad}}$. The constant $M_{\text{gs,rad}}$ is called \textbf{radial minimal mass}. 	
		\end{itemize}
	\end{remark}
	
	Using Definition $\ref{definition ground states}$, we have the following sharp Gagliardo-Nirenberg inequality: for $0<c<\lambda(d)$,
	\begin{align}
	\|u\|^{\frac{4}{d}+2}_{L^{\frac{4}{d}+2}} \leq C_{\text{GN}}(c) \|u\|^{\frac{4}{d}}_{L^2} \|u\|^2_{\dot{H}^1_c}, \label{sharp gagliardo-nirenberg modified}
	\end{align}
	for any $u \in H^1\backslash \{0\}$, where
	\[
	C_{\text{GN}}(c) = \frac{d+2}{d} \frac{1}{M^{\frac{4}{d}}_{\text{gs}}},
	\]
	and also the following sharp radial Gagliardo-Nirenberg inequality: for $c<0$,
	\begin{align}
	\|u\|^{\frac{4}{d}+2}_{L^{\frac{4}{d}+2}} \leq C_{\text{GN}}(c,\text{rad}) \|u\|^{\frac{4}{d}}_{L^2} \|u\|^2_{\dot{H}^1_c}, \label{sharp radial gagliardo-nirenberg modified}
	\end{align}
	for any $u \in H^1_{\text{rad}}\backslash \{0\}$, where
	\[
	C_{\text{GN}}(c, \text{rad}) = \frac{d+2}{d} \frac{1}{M^{\frac{4}{d}}_{\text{gs,rad}}}.
	\]
	
	We have the following variational structure of ground states.
	\begin{lemma} [Variational structure of ground states] \label{lemma variational structure ground states}
		\begin{itemize}
			\item Let $d\geq 3$ and $0<c<\lambda(d)$. If $v \in H^1$ satisfies
			\[
			\|v\|_{L^2} = M_{\emph{gs}}, \quad E(v) =0,
			\]
			then there exists $Q \in \mathcal{G}$ such that $v$ is of the form
			\[
			u(x) = e^{i\theta} \lambda^{\frac{d}{2}} Q(\lambda x),
			\]
			for some $\theta \in \R$ and $\lambda>0$. 
			\item Let $d\geq 3$ and $c<0$. If $v\in H^1_{\emph{rad}}$ satisfies
			\[
			\|v\|_{L^2} = M_{\emph{gs,rad}}, \quad E(v) =0,
			\]
			then there exists $Q_{\emph{rad}} \in \mathcal{G}_{\emph{rad}}$ such that $v$ is of the form
			\[
			v(x) = e^{i\vartheta} \rho^{\frac{d}{2}} Q_{\emph{rad}}(\rho x),
			\]
			for some $\vartheta \in \R$ and $\rho>0$.
		\end{itemize}
	\end{lemma}
	\begin{proof}
		In the case $0<c<\lambda(d)$, the proof of the above result is given in \cite[Proposition 3, p.119]{CsoboGenoud}. The one for $c<0$ is similar. We thus omit the details.
	\end{proof}
	\subsection{Limiting profile of finite time minimal mass blow-up solutions}
	Using the variational structure of ground states given in Lemma $\ref{lemma variational structure ground states}$, we obtain the following limiting profile of finite time blow-up solutions with minimal mass. This limiting profile plays a same role as the one proved by Csobo-Genoud in \cite[Proposition 4, p.120]{CsoboGenoud}. With the help of this limiting profile, we show the classification of finite time blow-up solutions with minimal mass for $(\ref{NLS inverse square})$.
	\begin{theorem}[Limiting profile with minimal mass] \label{theorem limiting profile minimal mass}
		\begin{itemize}
			\item Let $d\geq 3$ and $0<c<\lambda(d)$. Let $u_0 \in H^1$ be such that $\|u_0\|_{L^2}=M_{\emph{gs}}$. Suppose that the corresponding solution $u$ to $(\ref{NLS inverse square})$ blows up at finite time $0<T<+\infty$. Then for any time sequence $(t_n)_{n\geq 1}$ satisfying $t_n \uparrow T$, there exist a subsequence still denoted by $(t_n)_{n\geq 1}$, a function $Q\in \mathcal{G}$, sequences of $\theta_n \in \R$, $\lambda_n>0, \lambda_n \rightarrow 0$ and $x_n \in \R^d$ such that 
			\begin{align}
			e^{it\theta_n} \lambda_n^{\frac{d}{2}} u(t_n, \lambda_n \cdot + x_n) \rightarrow Q \text{ strongly in } H^1, \label{limiting profile positive c}
			\end{align}
			as $n\rightarrow \infty$.
			\item Let $d\geq 3$ and $c<0$. Let $u_0 \in H^1_{\emph{rad}}$ satisfy $\|u_0\|_{L^2} = M_{\emph{gs,rad}}$. Suppose that the corresponding solution $u$ to $(\ref{NLS inverse square})$ blows up at finite time $0<T<+\infty$. Then for any time sequence $(t_n)_{n\geq 1}$ satisfying $t_n \uparrow T$, there exist a subsequence still denoted by $(t_n)_{n\geq 1}$, a function $Q_{\emph{rad}}\in \mathcal{G}_{\emph{rad}}$, sequences of $\vartheta_n \in \R$ and $\rho_n>0, \rho_n\rightarrow 0$ such that
			\begin{align}
			e^{it\vartheta_n} \rho_n^{\frac{d}{2}} u(t_n, \rho_n \cdot ) \rightarrow Q_{\emph{rad}} \text{ strongly in } H^1, \label{limiting profile negative c}
			\end{align}
			as $n\rightarrow \infty$.
		\end{itemize}
	\end{theorem}
	
	\begin{proof}
		Let us firstly consider the case $0<c<\lambda(d)$. Let $(t_n)_{n\geq 1}$ be a sequence such that $t_n \uparrow T$. Set
		\[
		\lambda_n:= \frac{\|Q_c\|_{\dot{H}^1_c}}{\|u(t_n)\|_{\dot{H}^1_c}}, \quad v_n(x) := \lambda_n^{\frac{d}{2}} u(t_n, \lambda_n x),
		\]
		where $Q_c$ is given in Theorem $\ref{theorem sharp gagliardo-nirenberg}$. By the blow-up alternative, we see that $\lambda_n \rightarrow 0$ as $n \rightarrow \infty$. Moreover,
		\begin{align}
		\|v_n\|_{L^2} = \|u(t_n)\|_{L^2} = \|u_0\|_{L^2} = M_{\text{gs}}, \label{L2 norm v_n}
		\end{align}
		and
		\begin{align}
		\|v_n\|_{\dot{H}^1_c} = \lambda_n \|u(t_n)\|_{\dot{H}^1_c} = \|Q_c\|_{\dot{H}^1_c}, \label{H1c norm v_n}
		\end{align}
		and
		\[
		E(v_n) = \lambda_n^2 E(u(t_n)) = \lambda_n^2 E(u_0) \rightarrow 0,
		\]
		as $n \rightarrow \infty$. In particular, 
		\begin{align}
		\|v_n\|^{\frac{4}{d}+2}_{L^{\frac{4}{d}+2}}  \rightarrow \frac{d+2}{d} \|Q_c\|^2_{\dot{H}^1_c}, \label{lebesgue norm v_n}
		\end{align}
		as $n \rightarrow \infty$. Thus the sequence $(v_n)_{n\geq 1}$ satisfies conditions of Lemma $\ref{lemma compactness lemma}$ with 
		\[
		M^2 = \|Q_c\|^2_{\dot{H}^1_c}, \quad m^{\frac{4}{d}+2} = \frac{d+2}{d} \|Q_c\|^2_{\dot{H}^1_c}.
		\]
		Therefore, there exist $V \in H^1$ and a sequence $(x_n)_{n\geq 1}$ in $\R^d$ such that up to a subsequence,
		\[
		v_n(\cdot+ x_n) = \lambda_n^{\frac{d}{2}} u(t_n, \lambda_n \cdot+ x_n) \rightharpoonup V \text{ weakly in } H^1, 
		\]
		as $n\rightarrow \infty$ and $\|V\|_{L^2} \geq \|Q_c\|_{L^2} = M_{\text{gs}}$. Since $v_n(\cdot+x_n) \rightharpoonup V$ weakly in $H^1$ as $n\rightarrow \infty$, the semi-continuity of weak convergence and $(\ref{L2 norm v_n})$ imply
		\[
		M_{\text{gs}} \leq \|V\|_{L^2} \leq \liminf_{n\rightarrow \infty} \|v_n\|_{L^2} = M_{\text{gs}}.
		\]
		This shows that
		\begin{align}
		\|V\|_{L^2} = \lim_{n\rightarrow \infty} \|v_n\|_{L^2} = M_{\text{gs}}. \label{L2 convergence}
		\end{align}
		Therefore, 
		$
		v_n(\cdot+x_n) \rightarrow V \text{ strongly in } L^2
		$
		as $n\rightarrow \infty$. By the sharp Gagliardo-Nirenberg inequality $(\ref{sharp gagliardo-nirenberg modified})$, we also have that $v_n(\cdot+x_n) \rightarrow V$ strongly in $L^{\frac{4}{d}+2}$ as $n\rightarrow \infty$. Indeed, 
		\begin{align*}
		\|v_n(\cdot+x_n) - V\|^{\frac{4}{d}+2}_{L^{\frac{4}{d}+2}} &\leq C_{\text{GN}}(c) \|v_n(\cdot+ x_n) - V\|^{\frac{4}{d}}_{L^2} \|v_n(\cdot+x_n)-V\|^2_{\dot{H}^1_c} \\
		&\lesssim C_{\text{GN}}(c) \left(\|Q_c\|^2_{\dot{H}^1_c} + \|V\|^2_{\dot{H}^1_c} \right) \|v_n(\cdot+x_n) -V\|^{\frac{4}{d}}_{L^2} \rightarrow 0,
		\end{align*}
		as $n\rightarrow \infty$. Here we use 
		\[	
		\|v_n(\cdot+x_n)\|_{\dot{H}^1_x} \sim \|v_n(\cdot +x_n)\|_{\dot{H}^1} = \|v_n\|_{\dot{H}^1} \sim \|v_n\|_{\dot{H}^1_c}
		\]
		in the second estimate. Moreover, using $(\ref{lebesgue norm v_n})$, $(\ref{L2 convergence})$ and the sharp Gagliardo-Nirenberg inequality $(\ref{sharp gagliardo-nirenberg modified})$, we get
		\[
		\|Q_c\|^2_{\dot{H}^1_c} = \frac{d}{d+2} \lim_{n\rightarrow \infty} \|v_n\|^{\frac{4}{d}+2}_{L^{\frac{4}{d}+2}} = \frac{d}{d+2} \|V\|^{\frac{4}{d}+2}_{L^{\frac{4}{d}+2}} \leq \left(\frac{\|V\|_{L^2}}{M_{\text{gs}}} \right)^{\frac{4}{d}} \|V\|^2_{\dot{H}^1_c} = \|V\|^2_{\dot{H}^1_c}.
		\]
		Thus the semi-continuity of weak convergence and $(\ref{H1c norm v_n})$ imply
		\[
		\|Q_c\|_{\dot{H}^1_c} \leq \|V\|_{\dot{H}^1_c} \leq \liminf_{n\rightarrow \infty} \|v_n\|_{\dot{H}^1_c} = \|Q_c\|_{\dot{H}^1_c}.
		\]
		Hence 
		\[
		\|V\|_{\dot{H}^1_c} = \lim_{n\rightarrow \infty} \|v_n\|_{\dot{H}^1_c} = \|Q_c\|_{\dot{H}^1_c}.
		\]
		We next claim that 
		\[
		v_n(\cdot+x_n) \rightarrow V \text{ strongly in } \dot{H}^1,
		\]
		as $n \rightarrow \infty$. Since
		\[
		\|V\|^{\frac{4}{d}+2}_{L^{\frac{4}{d}+2}} = \frac{d+2}{d} \|Q_c\|^2_{\dot{H}^1_c}, \quad \|V\|_{\dot{H}^1_c}= \|Q_c\|_{\dot{H}^1_c},
		\] 
		we see that $E(V)=0$. It follows that there exists $V \in H^1$ such that 
		\[
		\|V\|_{L^2} = M_{\text{gs}}, \quad E(V)=0.
		\]
		The variational structure of ground states given in Lemma $\ref{lemma variational structure ground states}$ shows that there exists $Q \in \mathcal{G}$ such that $V(x) = e^{i\theta} \lambda^{\frac{d}{2}} Q(\lambda x)$ for some $\theta \in \R$ and $\lambda>0$. Thus,
		\[
		v_n(\cdot+x_n) = \lambda_n^{\frac{d}{2}} u(t_n, \lambda_n \cdot+ x_n) \rightarrow V= e^{i\theta} \lambda^{\frac{d}{2}} Q(\lambda \cdot) \text{ strongly in } H^1,
		\]
		as $n\rightarrow \infty$. Redefining $\tilde{\lambda}_n:= \lambda_n \lambda^{-1}$,
		we obtain 
		\[
		e^{-i\theta} \tilde{\lambda}_n^{\frac{d}{2}} u(t_n, \tilde{\lambda}_n \cdot + x_n) \rightarrow Q \text{ strongly in } H^1,
		\]
		as $n \rightarrow \infty$. We now prove the claim. Since $v_n(\cdot+x_n) \rightharpoonup V$ weakly in $H^1$. Set $r_n(x):= v_n(x) - V(x-x_n)$. We see that $r_n(\cdot+x_n) \rightharpoonup 0$ weakly in $H^1$. By the same argument as in the proof of Proposition $\ref{proposition profile decomposition}$, we have
		\[
		\int |x|^{-2} |v_n(x)|^2 dx = \int |x|^{-2} |V(x-x_n)|^2 dx + \int |x|^{-2} |r_n(x)|^2 dx +o_n(1).
		\]
		In particular, we have
		\begin{align*}
		\|v_n\|^2_{\dot{H}^1_c} &= \|V(\cdot-x_n)\|^2_{\dot{H}^1_c} + \|r_n\|^2_{\dot{H}^1_c} + o_n(1), \\
		\|v_n\|^{\frac{4}{d}+2}_{L^{\frac{4}{d}+2}} &= \|V(\cdot-x_n)\|^{\frac{4}{d}+2}_{L^{\frac{4}{d}+2}} + \|r_n\|^{\frac{4}{d}+2}_{L^{\frac{4}{d}+2}} +o_n(1),
		\end{align*}
		as $n\rightarrow \infty$. Thus,
		\begin{align}
		E(v_n) = E(V(\cdot-x_n)) + E(r_n) + o_n(1), \label{energy expansion}
		\end{align}
		as $n\rightarrow \infty$. On the other hand, since $v_n(\cdot+x_n) \rightarrow V$ strongly in $L^2$, it follows that $r_n(\cdot+x_n) \rightarrow 0$ strongly in $L^2$. This implies in particular that $r_n \rightarrow 0$ strongly in $L^2$ and $r_n \rightharpoonup 0$ weakly in $H^1$. The sharp Gagliardo-Nireberg inequality $(\ref{sharp gagliardo-nirenberg modified})$ then implies $r_n \rightarrow 0$ strongly in $L^{\frac{4}{d}+2}$. By the semi-continuity of weak convergence,
		\begin{align*}
		0 \leq \frac{1}{2} \liminf_{n\rightarrow \infty} \|r_n\|^2_{\dot{H}^1_c} &= \frac{1}{2} \liminf_{n\rightarrow \infty} \|r_n\|^2_{\dot{H}^1_c}  - \frac{d}{2d+4} \liminf_{n\rightarrow \infty} \|r_n\|_{L^{\frac{4}{d}+2}}^{\frac{4}{d}+2} \\
		&\leq \liminf_{n\rightarrow \infty} \left(\frac{1}{2}\|r_n\|^2_{\dot{H}^1_c} - \frac{d}{2d+4} \|r_n\|_{L^{\frac{4}{d}+2}}^{\frac{4}{d}+2} \right) \\
		&= \liminf_{n\rightarrow \infty} E(r_n).
		\end{align*}
		In particular,
		\begin{align*}
		\liminf_{n\rightarrow \infty} E(V(\cdot-x_n)) &\leq \liminf_{n\rightarrow \infty} E(V(\cdot-x_n)) + \liminf_{n\rightarrow \infty} E(r_n) \\
		&\leq \liminf_{n\rightarrow \infty} (E(V(\cdot-x_n)) + E(r_n)) = \liminf_{n\rightarrow \infty} E(v_n) =0.
		\end{align*}
		We also have from the sharp Gagliardo-Nirenberg inequality $(\ref{sharp gagliardo-nirenberg modified})$ and the fact $\|V(\cdot-x_n)\|_{L^2}= \|V\|_{L^2} = M_{\text{gs}}$ that $E(V(\cdot-x_n)) \geq 0$ for all $n\geq 1$. Therefore, we must have
		\[
		\liminf_{n\rightarrow \infty} E(V(\cdot-x_n)) =0.
		\]
		Taking $\liminf$ both sides of $(\ref{energy expansion})$, we obtain 
		$\liminf_{n\rightarrow \infty} E(r_n)=0$. Since $r_n \rightarrow 0$ strongly in $L^{\frac{4}{d}+2}$, we see that up to a subsequence, $\lim_{n\rightarrow \infty} \|r_n\|_{\dot{H}^1_c}=0$. Using the equivalence $\|\cdot\|_{\dot{H}^1_c} \sim \|\cdot\|_{\dot{H}^1}$, we obtain $\lim_{n\rightarrow \infty} \|\nabla r_n\|_{L^2}=0$.
		Thanks to the expansion
		\[
		\|\nabla v_n\|^2_{L^2} = \|\nabla V\|^2_{L^2} + \|\nabla r_n\|^2_{L^2} +o_n(1),
		\]
		as $n\rightarrow \infty$, we obtain
		\[
		\lim_{n\rightarrow \infty} \|\nabla v_n\|_{L^2} = \|\nabla V\|_{L^2}.
		\]
		Since $v_n(\cdot+x_n) \rightharpoonup V$ weakly in $H^1$, we infer that $v_n(\cdot+x_n) \rightarrow V$ strongly in $\dot{H}^1$ as $n\rightarrow \infty$. This proves the claim and the proof of the first item is complete.
		
		We now consider the case $c<0$. Let $(t_n)_{n\geq 1}$ be a sequence such that $t_n \uparrow T$. Denote
		\[
		\rho_n:= \frac{\|Q_{c,\text{rad}}\|_{\dot{H}^1_c}}{\|u(t_n)\|_{\dot{H}^1_c}}, \quad v_n(x):= \rho_n^{\frac{d}{2}} u(t_n, \rho_n x),
		\]
		where $Q_{c,\text{rad}}$ is given in Theorem $\ref{theorem sharp gagliardo-nirenberg}$. Since $u_0 \in H^1_{\text{rad}}$, we see that $u(t) \in H^1_{\text{rad}}$ for any $t$ as long as the solution exists. By the blow-up alternative, it follows that $\rho_n \rightarrow 0$ as $n\rightarrow \infty$. We also have
		\begin{align}
		\|v_n\|_{L^2} = \|u(t_n)\|_{L^2} = \|u_0\|_{L^2} = M_{\text{gs,rad}}, \quad \|v_n\|_{\dot{H}^1_c} = \rho_n\|u(t_n)\|_{\dot{H}^1_c} = \|Q_{c,\text{rad}}\|_{\dot{H}^1_c}, \label{property v_n radial}
		\end{align}
		and
		\[
		E(v_n) = \rho_n^2 E(u(t_n)) = \rho^2_n E(u_0) \rightarrow 0,
		\]
		as $n\rightarrow \infty$. In particular,
		\begin{align}
		\|v_n\|^{\frac{4}{d}+2}_{L^{\frac{4}{d}+2}} \rightarrow \frac{d+2}{d} \|Q_{c,\text{rad}}\|^2_{\dot{H}^1_c}, \label{limiting profile radial proof 1}
		\end{align}
		as $n\rightarrow \infty$. We thus obtain a bounded sequence $(v_n)_{n\geq 1}$ of $H^1_{\text{rad}}$-functions which satisfies conditions of Lemma $\ref{lemma compactness lemma}$ with 
		\[
		M^2= \|Q_{c,\text{rad}}\|^2_{\dot{H}^1_c}, \quad m^{\frac{4}{d}+2} =\frac{d+2}{d} \|Q_{c,\text{rad}}\|^2_{\dot{H}^1_c}.
		\]
		We learn from Lemma $\ref{lemma compactness lemma}$ that there exists $V \in H^1_{\text{rad}}$ such that up to a subsequence
		\[
		v_n \rightharpoonup V \text{ weakly in } H^1,
		\]
		as $n \rightarrow \infty$ and $\|V\|_{L^2} \geq \|Q_{c,\text{rad}}\|_{L^2} = M_{\text{gs,rad}}$. The semi-continuity of weak convergence and $(\ref{property v_n radial})$ imply that 
		\begin{align}
		M_{\text{gs,rad}} \leq \|V\|_{L^2} \leq \liminf_{n\rightarrow \infty} \|v_n\|_{L^2} = M_{\text{gs,rad}}. \label{limiting profile radial proof 2}
		\end{align}
		We thus get
		\[
		\|V\|_{L^2} = \lim_{n\rightarrow \infty} \|v_n\|_{L^2} = M_{\text{gs,rad}}.
		\]
		In particular, $v_n \rightarrow V$ strongly in $L^2$ as $n\rightarrow \infty$. This together with the sharp Gagliardo-Nirenberg inequality $(\ref{sharp radial gagliardo-nirenberg modified})$ yield that $v_n \rightarrow V$ strongly in $L^{\frac{4}{d}+2}$ as $n\rightarrow \infty$. By $(\ref{limiting profile radial proof 1})$ and $(\ref{limiting profile radial proof 2})$, the sharp Gagliardo-Nirenberg inequality implies that
		\[
		\|Q_{c,\text{rad}}\|^2_{\dot{H}^1_c} =\frac{d}{d+2} \lim_{n\rightarrow \infty} \|v_n\|^{\frac{4}{d}+2}_{L^{\frac{4}{d}+2}} = \frac{d}{d+2} \|V\|^{\frac{4}{d}+2}_{L^{\frac{4}{d}+2}} \leq \left(\frac{\|V\|_{L^2}}{M_{\text{gs,rad}}} \right)^{\frac{4}{d}} \|V\|^2_{\dot{H}^1_c} = \|V\|^2_{\dot{H}^1_c}.
		\]
		Using the above inequality, the semi-continuity of weak convergence and $(\ref{property v_n radial})$ imply
		\[
		\|Q_{c,\text{rad}}\|_{\dot{H}^1_c} \leq \|V\|_{\dot{H}^1_c} \leq \liminf_{n\rightarrow \infty} \|v_n\|_{\dot{H}^1_c} = \|Q_{c,\text{rad}}\|_{\dot{H}^1_c}.
		\]
		Hence
		\[
		\|V\|_{\dot{H}^1_c} = \lim_{n\rightarrow \infty} \|v_n\|_{\dot{H}^1_c} = \|Q_{c,\text{rad}}\|_{\dot{H}^1_c}.
		\]
		We now claim that 
		\[
		v_n \rightarrow V \text{ strongly in } H^1,
		\]
		as $n\rightarrow \infty$. To see this, we write
		\[
		v_n(x) = V(x) +r_n(x),
		\]
		with $r_n \rightharpoonup 0$ weakly in $H^1$ as $n\rightarrow \infty$. We easily verify that
		\[
		E(v_n) = E(V) + E(r_n) + o_n(1),
		\]
		as $n\rightarrow \infty$. Since $v_n \rightarrow V$ strongly in $L^2$, we see that $r_n \rightarrow 0$ strongly in $L^2$. The sharp Gagliardo-Nirenberg inequality $(\ref{sharp radial gagliardo-nirenberg modified})$ then implies that $r_n \rightarrow 0$ strongly in $L^{\frac{4}{d}+2}$. Arguing as in the case $0<c<\lambda(d)$, we get
		\[
		\liminf_{n\rightarrow \infty} E(r_n) \geq 0,
		\]
		and
		\[
		E(V) \leq E(V) + \liminf_{n\rightarrow \infty} E(r_n) \leq \liminf_{n\rightarrow \infty} (E(V) + E(r_n)) =\liminf_{n\rightarrow \infty} E(v_n) =0.
		\]
		On the other hand, since $\|V\|_{L^2} = M_{\text{gs,rad}}$, the sharp Gagliardo-Nirenberg inequality $(\ref{sharp radial gagliardo-nirenberg modified})$ implies that $E(V) \geq 0$. Therefore, $E(V)=0$. As a result, we obtain that
		\[
		\liminf_{n\rightarrow \infty} E(r_n) =0.
		\]
		Since $r_n \rightarrow 0$ strongly in $L^{\frac{4}{d}+2}$, we see that up to a subsequence, $\lim_{n\rightarrow \infty} \|r_n\|_{\dot{H}^1_c}=0$.  This implies in particular that $\lim_{n\rightarrow \infty} \|\nabla r_n\|_{L^2} =0$. Using the fact
		\[
		\|\nabla v_n\|_{L^2}^2 = \|\nabla V\|^2_{L^2} + \|\nabla r_n\|^2_{L^2} +o_n(1),
		\]
		as $n\rightarrow \infty$, we obtain $\lim_{n\rightarrow \infty} \|\nabla v_n\|_{L^2} = \|\nabla V\|_{L^2}$. Since $v_n \rightharpoonup V$ weakly in $H^1$ as $n\rightarrow \infty$, it follows that $v_n \rightarrow V$ strongly in $\dot{H}^1$ as $n\rightarrow \infty$. This proves the claim. 
		
		We thus obtain $V \in H^1_{\text{rad}}$ such that 
		\[
		\|V\|_{L^2} = M_{\text{gs,rad}}, \quad E(V)=0.
		\]
		The second equality follows from 
		\[
		\|V\|^{\frac{4}{d}+2}_{L^{\frac{4}{d}+2}} = \frac{d+2}{d} \|Q_{c,\text{rad}}\|^2_{\dot{H}^1_c}, \quad \|V\|_{\dot{H}^1_c} = \|Q_{c,\text{rad}}\|_{\dot{H}^1_c}.
		\]
		The variational structure of radial ground states given in Lemma $\ref{lemma variational structure ground states}$ implies that there exists $Q_{\text{rad}} \in \mathcal{G}_{\text{rad}}$ such that $V(x) = e^{i\vartheta} \rho^{\frac{d}{2}} Q(\rho x)$ for some $\vartheta \in \R$ and $\rho >0$. We thus obtain
		\[
		v_n(\cdot) = \rho^{\frac{d}{2}}_n u(t_n, \rho_n \cdot) \rightarrow V = e^{i\vartheta} \rho^{\frac{d}{2}} Q_{\text{rad}}(\rho \cdot) \text{ strongly in } H^1,
		\]
		as $n\rightarrow \infty$. Redefining $\tilde{\rho}_n:= \rho_n \rho^{-1}$, we obtain 
		\[
		e^{-i\vartheta} \tilde{\rho}_n^{\frac{d}{2}} u(t_n, \tilde{\rho}_n \cdot) \rightarrow Q_{\text{rad}} \text{ strongly in } H^1,
		\]
		as $n\rightarrow \infty$. The proof is complete.
	\end{proof}
	
	In order to prove the characterization of finite time blow-up solutions with minimal mass, we need to recall basic facts related to $(\ref{NLS inverse square})$. Let us start with the following Cauchy-Schwarz inequality due to Banica \cite{Banica}.
	\begin{lemma} \label{lemma cauchy-schwarz inequality}
		If one of the following conditions holds true
		\begin{itemize}
			\item $d\geq 3$, $0<c<\lambda(d)$ and $u\in H^1$ is such that $\|u\|_{L^2} = M_{\emph{gs}}$,
			\item $d\geq 3$, $c<0$ and $u \in H^1_{\emph{rad}}$ is such that $\|u\|_{L^2} = M_{\emph{gs,rad}}$,
		\end{itemize}
		then for any real valued function $\varphi \in C^1$ satisfying $\nabla \varphi$ is bounded, we have
		\begin{align}
		\left|\int \nabla \varphi \cdot \imem{(u \nabla \overline{u})} dx \right| \leq \sqrt{2E(u)} \left(\int |\nabla \varphi|^2 |u|^2 dx \right)^{1/2}. \label{cauchy-schwarz inequality}
		\end{align}
	\end{lemma}
	Note that by sharp Gagliardo Nirenberg inequalities $(\ref{sharp gagliardo-nirenberg modified})$ and $(\ref{sharp radial gagliardo-nirenberg modified})$, the above assumptions imply $E(u)$ is non-negative. 
	
	We also need the following virial identity (see e.g. \cite[Lemma 5.3]{Dinh-inverse} or \cite[Lemma 3, p.124]{CsoboGenoud}).
	\begin{lemma}[Virial identity] \label{lemma virial identity}
		Let $d\geq 3$ and $c \ne 0$ be such that $c<\lambda(d)$. Let $u_0 \in H^1$ be such that $|x| u_0 \in L^2$ and $u: I \times \R^d \rightarrow \C$ the corresponding solution to $(\ref{NLS inverse square})$. Then $|x| u \in C(I, L^2)$ and for any $t\in I$,
		\begin{align}
		\frac{d^2}{dt^2} \|x u(t)\|^2_{L^2} = 16 E(u_0). \label{virial identity}
		\end{align}
		In particular, we have for any $t\in I$,
		\begin{align}
		\begin{aligned}
		\int |x|^2 |u(t)|^2 dx &= \int |x|^2 |u_0|^2 dx - 4t \int x \cdot \imem{(u_0 \nabla \overline{u}_0)} dx + 8t^2 E(u_0)\\
		&= 8t^2 E\left(e^{i\frac{|x|^2}{4t}} u_0\right).
		\end{aligned}
		\label{virial identity application}
		\end{align}
	\end{lemma}
	\begin{proof}
		We refer the reader to \cite[Lemma 5.3]{Dinh-inverse} or \cite[Lemma 3, p.124]{CsoboGenoud} for the proof of $(\ref{virial identity})$. The first identity in $(\ref{virial identity application})$ follows by integrating $(\ref{virial identity})$ over the time $t$. The second identity in $(\ref{virial identity application})$ follows from a direct computation using the fact that
		\[
		\left| \nabla \left(e^{i\frac{|x|^2}{4t}} u_0 \right) \right| = \frac{1}{4t^2} |x|^2 |u_0|^2 - \frac{1}{t} x\cdot \im{(u_0 \nabla \overline{u}_0)} +|\nabla u_0|^2.
		\]
		The proof is complete.
	\end{proof}
	
	We are now able to prove the characterization of finite time blow-up solutions with minimal mass given in Theorem $\ref{theorem characterization minimal mass}$.
	
	\noindent \textit{Proof of Theorem $\ref{theorem characterization minimal mass}$.}  
	Let us firstly consider the case $0<c<\lambda(d)$. Let $(t_n)_{n\geq 1}$ be such that $t_n\uparrow T$. By Theorem $\ref{theorem limiting profile minimal mass}$, we see that up to a subsequence, there exists $Q \in \mathcal{G}$ such that
	\begin{align}
	e^{i\theta_n} \lambda_n^{\frac{d}{2}} u(t_n, \lambda_n \cdot+ x_n) \rightarrow Q \text{ strongly in } H^1, \label{strong convergence}
	\end{align}
	as $n \rightarrow \infty$, where $(\theta_n)_{n\geq 1} \subset \R, (x_n)_{n\geq 1} \subset \R^d$ and $\lambda_n \rightarrow 0$ as $n\rightarrow \infty$. From this, we infer that
	\begin{align}
	|u(t_n,x)|^2 dx - \|Q\|_{L^2}^2 \delta_{x=x_n} \rightharpoonup 0, \label{weak convergence measure}
	\end{align}
	as $n\rightarrow \infty$. 
	
	Up to subsequence, we may assume that $x_n \rightarrow x_0 \in \{0,\infty\}$. Now let $\varphi$ be a smooth non-negative radial compactly supported function satisfying
	\[
	\varphi(x)=|x|^2 \text{ if } |x|<1, \text{ and } |\nabla \varphi(x)|^2 \leq C\varphi(x),
	\]
	for some constant $C>0$. For $R>1$, we define 
	\[
	\varphi_R(x):= R^2 \varphi(x/R), \quad U_R(t):= \int \varphi_R(x) |u(t,x)|^2 dx.
	\]
	Using the Cauchy-Schwarz inequality $(\ref{cauchy-schwarz inequality})$ and the fact $|\nabla \varphi_R|^2 \leq C |\varphi_R|$, we have
	\begin{align*}
	|U'_R(t)| &= 2\left| \int \nabla \varphi_R \cdot \im{ (u(t) \nabla \overline{u}(t))} dx \right| \\
	&\leq 2 \sqrt{2E(u_0)} \left(\int |u(t)|^2 |\nabla \varphi_R|^2 dx \right)^{1/2} \\
	&\leq C(u_0) \sqrt{U_R(t)}. 
	\end{align*}
	Integrating with respect to $t$, we obtain
	\begin{align}
	\left|\sqrt{U_R(t)} -\sqrt{U_R(t_n)} \right| \leq C(u_0) |t_n-t|. \label{estimate big U_R}
	\end{align}
	Thanks to $(\ref{weak convergence measure})$, we see that $U_R(t_n) \rightarrow 0$ as $n\rightarrow \infty$. Indeed, if $|x_n| \rightarrow 0$, then $U_R(t_n) \rightarrow \|Q\|^2_{L^2} \varphi_R(0) =0$ as $n\rightarrow \infty$. If $|x_n| \rightarrow \infty$, then $U_R(t_n) \rightarrow 0$ since $\varphi_R$ is compactly supported. Letting $n\rightarrow \infty$ in $(\ref{estimate big U_R})$, we obtain
	\[
	U_R(t) \leq C(u_0) (T-t)^2.
	\]
	Now fix $t \in [0,T)$, letting $R\rightarrow \infty$, we have
	\begin{align}
	8t^2 E\left(e^{i\frac{|x|^2}{4t}} u_0 \right) = \int |x|^2 |u(t,x)|^2 dx \leq C(u_0) (T-t)^2, \label{limit virial action}
	\end{align}
	where the first equality follows from Lemma $\ref{lemma virial identity}$. Note that we have from $(\ref{limit virial action})$ that $u(t) \in L^2(|x|^2dx)$ for any $t\in [0,T)$. We also have from $(\ref{weak convergence measure})$ and $(\ref{limit virial action})$ that
	\[
	\liminf_{n\rightarrow \infty} |x_n|^2 \|Q\|^2_{L^2} \leq C(u_0) T^2.
	\]
	Thus $x_n$ cannot go to infinity, hence $x_n$ converges to zero. Letting $t$ tends to $T$, we learn from $(\ref{limit virial action})$ that
	\[
	E\left(e^{i\frac{|x|^2}{4T}} u_0 \right) =0.
	\]
	We also have 
	\[
	\Big\| e^{i\frac{|x|^2}{4T}} u_0 \Big\|_{L^2} = \|u_0\|_{L^2} = M_{\text{gs}}.
	\]
	Therefore, Lemma $\ref{lemma variational structure ground states}$ shows that there exists $\tilde{Q} \in \mathcal{G}$ such that
	\[
	e^{i\frac{|x|^2}{4T}} u_0(x) = e^{i\tilde{\theta}} \tilde{\lambda}^{\frac{d}{2}} \tilde{Q}(\tilde{\lambda} x).
	\]
	Redefining $\tilde{\lambda} = \frac{\lambda}{T}$ and $\tilde{\theta} = \theta + \frac{\lambda^2}{T}$, we obtain
	\[
	u_0(x) = e^{i\theta} e^{i\frac{\lambda^2}{T}} e^{-i\frac{|x|^2}{4T}} \left( \frac{\lambda}{T}\right)^{\frac{d}{2}} \tilde{Q}\left( \frac{\lambda x}{T}\right).
	\]
	This shows $(\ref{characterization initial data})$. By the uniqueness of solution to $(\ref{NLS inverse square})$, we find that $u(t) = S_{\tilde{Q}, T, \theta, \lambda}(t)$ for any $t\in [0,T)$. This completes the proof of the case $0<c<\lambda(d)$.
	
	Let us now consider the case $c<0$. Let $(t_n)_{n\geq 1}$ be such that $t_n\uparrow T$. We have from Theorem $\ref{theorem limiting profile minimal mass}$ that up to a subsequence, there exists $Q_{\text{rad}}\in \mathcal{G}_{\text{rad}}$ such that
	\[
	e^{i\vartheta_n} \rho_n^{\frac{d}{2}} u(t_n, \rho_n \cdot) \rightarrow Q_{\text{rad}} \text{ strongly in } H^1,
	\]
	as $n\rightarrow \infty$, where $(\vartheta_n)_{n\geq 1} \subset \R$ and $\rho_n\rightarrow 0$ as $n\rightarrow \infty$. This implies that
	\begin{align}
	|u(t_n,x)|^2dx - \|Q_{\text{rad}}\|^2_{L^2} \delta_{x=0} \rightharpoonup 0, \label{measure convergence radial}
	\end{align}
	as $n\rightarrow \infty$. By the same argument as in the case $0<c<\lambda(d)$, we learn that
	\[
	\left| \sqrt{U_R(t)} - \sqrt{U_R(t_n)}\right| \leq C(u_0) |t_n-t|.
	\]
	Here $U_R(t_n) \rightarrow 0$ as $n\rightarrow \infty$. Indeed, by $(\ref{measure convergence radial})$, $U_R(t_n) \rightarrow \|Q_{\text{rad}}\|^2_{L^2} \varphi_R(0)=0$ as $n\rightarrow \infty$. Therefore, letting $n\rightarrow \infty$, we obtain
	\[
	U_R(t) \leq C(u_0) (T-t)^2.
	\]
	Fix $t\in [0,T)$, letting $R \rightarrow \infty$, we obtain
	\[
	8t^2 E\left( e^{i\frac{|x|^2}{4t}} u_0\right) = \int |x|^2 |u(t,x)|^2 dx \leq C(u_0) (T-t)^2.
	\]
	Letting $t\uparrow T$, we get 
	\[
	E\left( e^{i\frac{|x|^2}{4T}} u_0\right) =0,
	\]
	and also
	\[
	\left\| e^{i\frac{|x|^2}{4T}} u_0\right\|_{L^2} = \|u_0\|_{L^2} = M_{\text{gs,rad}}.
	\]
	By Lemma $\ref{lemma variational structure ground states}$, there exists $\tilde{Q}_{\text{rad}} \in \mathcal{G}_{\text{rad}}$ such that 
	\[
	e^{i\frac{|x|^2}{4T}} u_0(x) = e^{i\tilde{\vartheta}} \tilde{\rho}^{\frac{d}{2}} \tilde{Q}_{\text{rad}}(\tilde{\rho} x).
	\]
	Redefining $\tilde{\rho} = \frac{\rho}{T}$ and $\tilde{\vartheta} = \vartheta + \frac{\rho^2}{T}$, we get
	\[
	u_0 = e^{i\vartheta} e^{i\frac{\rho^2}{T}} e^{-i\frac{|x|^2}{4T}} \left( \frac{\rho}{T}\right)^{\frac{d}{2}} \tilde{Q}_{\text{rad}}\left(\frac{\rho x}{T} \right).
	\]
	This shows $(\ref{characterization initial data radial})$. The proof is complete.
	\defendproof
	
	\section*{Acknowledgments}
	A. Bensouilah would like to thank his thesis advisor, Pr. Sahbi Keraani, for many helpful discussions about the NLS with inverse-square potential, as well as his constant encouragement. V. D. Dinh would like to express his deep gratitude to his wife-Uyen Cong for her encouragement and support. He also would like to thank Prof. Changxing Miao for a discussion about Pohozaev's identities. The authors would like to thank the reviewers for their helpful comments and suggestions.

\end{document}